\documentclass[a4paper,11pt]{amsart}
\usepackage[british]{babel}
\usepackage[colorlinks=true,citecolor=black,linkcolor=black]{hyperref}
\usepackage{cite,microtype,amssymb,nicefrac}
\newtheorem{theorem}{Theorem}[section]
\newtheorem{lemma}[theorem]{Lemma}
\newtheorem{assumption}[theorem]{Assumption}
\newtheorem*{lemma*}{Lemma}
\theoremstyle{remark}
\newtheorem{definition}[theorem]{Definition}
\newtheorem{remark}[theorem]{Remark}

\newcommand{\abs}[1]{\left|#1\right|} % Absolute value
\newcommand{\cB}{\mathcal{B}} % Banach space
\newcommand{\cC}{\mathcal{C}} % Differentiable
\newcommand{\bC}{\mathbb{C}} % Complex numbers
\newcommand{\cD}{\mathcal{D}} % Distributions
\newcommand{\cS}{\mathcal{S}} % Schwartz space
\newcommand{\cJ}{\mathcal{J}} % Invariant distributions
\newcommand{\cL}{\mathcal{L}} % Transfer operator
\newcommand{\bN}{\mathbb{N}} % Natural numbers
\newcommand{\bR}{\mathbb{R}} % Reals
\newcommand{\bT}{\mathbb{T}} % Torus
\newcommand{\bZ}{\mathbb{Z}} % Integers
\newcommand{\norm}[1]{\left\|#1\right\|} % Norm
\newcommand{\snorm}[1]{\left|#1\right|} % Seminorm
\newcommand{\littleo}[1]{\operatorname{\mathnormal{o}}(#1)} % Little o
\newcommand{\bH}{\mathbb{H}} % Heisenberg group
\newcommand{\dv}{V} % Horizontal vector field
\newcommand{\dw}{W} % Vertical vector field
\newcommand{\dz}{Z} % Central vector field
\newcommand{\spec}[1]{\operatorname{Spec}(#1)} % Spectrum
\newcommand{\lquotient}[2]{\raisebox{-.2em}{\(#1\)} \diagdown #2} % Left quotient
\begin{document}

\title{Anisotropic spaces and nil-automorphisms}

\author{Oliver Butterley}
\address{(Oliver Butterley) Department of Mathematics, University of Rome Tor Vergata, Via della Ricerca Scientifica, 00133 Roma, Italy}
\email{butterley@mat.uniroma2.it}
\thanks{$^\ast$Corresponding author.}

\author{Minsung Kim$^{\ast}$}
\address{(Minsung Kim) Department of Mathematics, Pohang University of Science and Technology (POSTECH), 77 Cheongam-Ro, Nam-Gu, Pohang, Gyeongbuk, 37673 Korea}
\email{minsung@postech.ac.kr}

\begin{abstract}
    We introduce a family of geometric anisotropic Banach spaces on Heisenberg nilmanifolds and study the spectrum of the composition operator associated to partially hyperbolic automorphisms.
    Choosing amongst the family of Banach spaces, it is possible to make the essential spectral radius arbitrarily small.
    We show that the exterior part of the discrete spectrum coincides with the spectrum restricted to the kernel of one of the operators associated to the nil-automorphism.
    Moreover we show that the remainder of the discrete spectrum is self-similar, it is given by scaled copies of the exterior part.
\end{abstract}

\maketitle
\thispagestyle{empty}

%%%%%%%%%%%%%%%%%%%%%%
\section{Introduction and results}
\label{sec:results}
%%%%%%%%%%%%%%%%%%%%%%

% Transfer operators and spectral theory for dynamical systems
Transfer operators are used widely in the study of dynamical systems. 
The combination of functional analytic techniques and dynamical systems theory gives powerful methods to construct invariant measures as well as a means to study decay of correlation and other statistical properties. 
This idea goes back, at least, to the Koopman operator, and its use by von Neumann to prove the mean ergodic theorem.
Some milestones of the development include:
Exploring the spectral theory for the Koopman operator and its relation to the statistical properties of the system (ergodicity, mixing, etc.)~\cite{CFS82};
Studying the adjoint of the Koopman operator, sometimes called the Perron-Frobenius or the Ruelle-Perron-Frobenius transfer operator;
Avoiding the need for coding the system~\cite{Bowen75} but instead directly study the transfer operator acting on density functions (e.g.,~\cite{LY73, Ruelle89}), not only for expanding systems but also for the more involved hyperbolic case.
The theory is now rather well developed (see, e.g., \cite{Keller98, Baladi00, Baladi16}) and has resulted in multiple significant breakthroughs in the understanding of diverse dynamical systems. 
Despite this advanced state of development, many details remain to be fully understood.

% Anisotropic spaces
The work of Blank, Keller \& Liverani~\cite{BKL02}, showed that it was possible to directly study the transfer operators associated to hyperbolic systems by taking advantage of anisotropic Banach spaces.
In the years since then, many of these anisotropic Banach spaces have been constructed, capturing the behaviour of rather general hyperbolic systems (e.g., \cite{GL06, GL08, BT07, BT08, FRS08}, and \cite{Baladi16, Baladi17} for an overview).

%Systems with neutral direction
Understanding the statistical properties of partially hyperbolic systems (i.e., with a neutral direction) is hard.
(See, e.g., \cite{Dolgopyat98,BL07,BL13,Butterley16,ABV16,BE17,BW20,CL22,TZ23}.)
It appears essential to perform some form of oscillatory cancellation argument at some stage of the process in order to deal with the direction that sees neither expansion nor contraction.
There has been substantial progress on this topic in recent years but there remain open questions and it is desirable to further develop the technology for studying such systems.

% Identifying point spectrum
Until recently attention was mostly directed on: (1) the peripheral spectrum since this encodes quantitative information on ergodicity and mixing and yields the rate of mixing (e.g, \cite{Liverani95,Baladi00}); (2) estimating the essential spectral radius and connecting this to the meromorphic domain of the zeta-function \cite{Baladi16}. 
For systems with a neutral direction, either flows or partially hyperbolic systems, the process is more involved but there is still the possibility of obtaining similar spectral information (e.g., \cite{Faure07,FT13,FT15,GHW16,FT17,GF18}).
In recent years it has become clear that it is both possible and desirable to go beyond this and to obtain a detailed understanding of the point spectrum~\cite{FT17, DZ16, DZ17, GF18, GL19, FGL19, BS20, KW20, Baladi22, BKL22, GGHW24, BCJ22, BCC24}.
In the case of pseudo-Anosov maps, the work of Faure, Gouëzel \& Lanneau showed that the point spectrum can be identified using a connection to the action of the dynamics on cohomology~\cite{FGL19}.
In other cases it was possible to obtain results related to bands of spectrum for transfer operators associated to systems with a neutral direction~\cite{Faure07,FT13,FT15,FT17, BS20} and closely related is the semiclassical analysis for contact Anosov flows~\cite{FT13,GHW16}.
In a slightly different direction, other works investigated the possibility of an explicit description of the spectrum for analytic expanding or hyperbolic maps (Blaschke products)~\cite{SBJ13, BJS17, SBJ17}.
In yet another direction, perturbation and genericity results were obtained~\cite{KR04, Naud12, Adam17, BN19}.

% Renormalisation for parabolic dynamics
Heisenberg nilflows are key examples of parabolic\footnote{Parabolic in the sense that the distance between two close flow trajectories grows polynomially with time. Parabolic behavior is often characterized by zero entropy, lying strictly between hyperbolic and elliptic.} dynamics.
As such they are also key examples for the use of renormalization techniques.
The flows on general (higher-step) nilmanifolds are not renormalizable due to the complexity of commutation relations on Lie algebras. 
Nevertheless, new methods were introduced on certain types (so called Quasi-abelian or Triangular type) to move beyond renormalization methods \cite{Forni16,FF23,Kim22b}.
The renormalization scheme for Heisenberg nilflows was studied by Flaminio \& Forni~\cite{FF06} and they proved results on the deviation of ergodic averages (for higher rank actions see \cite{CF15, Kim22a}).  
In our case, the automorphism we study corresponds to a periodic type of renormalization cocycle (see \eqref{eq:renorm}).
As such, once the spectrum of the transfer operator has been determined, it can be used to prove the deviation of ergodic averages by adopting the ideas of Giulietti \& Liverani~\cite{GL19} (see also \cite{BS20,AB22,Galli24}). 
A possible benefit of this recent push of the techniques would be to allow results previously only available for the algebraic systems to be extended to general systems.

%Aim
This work fits amongst these above mentioned topics.
We continue to explore the possibilities and refine the technology, in particular we study a class of specific partially hyperbolic systems (nil-automorphisms on Heisenberg manifolds), develop a family of anisotropic Banach spaces amenable to the present setting and obtain complete spectral data for the associated transfer operator. 

% Heisenberg group
Let \(\bH\) be the three-dimensional Heisenberg group.
Up to isomorphism, \(\bH\) is the group of upper triangular matrices
\[
    \left\{
    \left(
    \begin{smallmatrix}
            1 & x & z \\
            0 & 1 & y \\
            0 & 0 & 1
        \end{smallmatrix}
    \right)
    : x,y,z \in \bR
    \right\}
\]
with the group law being usual matrix multiplication.
Equivalently, \(\bH\) is equal to \(\bR^3\) with the group law
\[
    (x,y,z) * (x',y',z') = (x+x', y+y', z+z' + xy').
\]
The corresponding Lie algebra has a basis \(\{\dv,\dw,\dz\}\), which satisfies the following commutation relations:
\begin{equation}
    \label{eq:comm}
    [\dv,\dw]=\dz, \quad [\dv,\dz]=[\dw,\dz]=0.
\end{equation}
(For further details see \cite[Chapter 1]{Folland89} where the version presented here is called the \emph{polarised  Heisenberg group}.)

% Heisenberg nilmanifold
A Heisenberg nilmanifold is the compact quotient space \(M = \lquotient{\Gamma}{\bH}\) where \(\Gamma\) is a discrete subgroup of \(\bH\).
Up to an automorphism of \(\bH\), every such subgroup \(\Gamma\) is of the form 
\begin{equation}
    \label{eq:gamma-K}
    \Gamma_K = \left\{(x, y, z / K) : x,y,z \in \bZ \right\} \subset \bH
\end{equation} 
for some\footnote{We use the convention that \(\bN = \{1,2,3,\ldots\}\) and \(\bN_0 = \{0,1,2,\ldots\}\).} \(K \in \bN\)~\cite[Theorem 1.10]{Tolimieri78}.
%Note that \(M\) is a circle bundle over the 2-torus \(\mathbb{T}^2 = \rquotient{\bR^2}{\bZ^2}\) with fibres given by one-parameter subgroup \(\{\exp (t Z)\}_{t \in \bR}\). 

% There exists an isomorphism between the set of Heisenberg frame and the automorphisms of \(\bH\), denoted by $\operatorname{Aut}_Z(\bH)$, which fix the central vector field $Z$. This induces an automorphism on nilmanifold $M$ as well. 
% partially hyperbolic automorphisms
Note that the \(\dz\)-fibres in the quotient space \(M\) are circles. 
The object of our interest is \(\Phi : M \to M\), an automorphism such that, for some \(\lambda > 1\),
\begin{equation}
    \label{eq:part-hyp}
    \Phi_* \dv = \lambda^{-1} \dv, \quad
    \Phi_* \dw = \lambda \dw, \quad
    \Phi_* \dz = \dz.
\end{equation}
We say that such an automorphism is \emph{partially hyperbolic with neutral centre} since it exhibits contraction, expansion and neutral behaviour.
In Appendix~\ref{sec:examples} we give details of this connection with toral automorphisms.
Note that, as explained there, one could reasonably take the point of view that the automorphism comes first and \(\{\dv,\dw,\dz\}\) and \(\lambda\) are defined by the automorphism.
Let \(E_\dv\), \(E_\dw\), \(E_\dz\) denote the one dimensional sub-bundles of tangent space which correspond to \(\dv\), \(\dw\), \(\dz\) respectively.
Consequently the tangent bundle admits a splitting \(TM = E_\dv \oplus E_\dw \oplus E_\dz\) such that \(\Phi\) is
\begin{itemize}
    \item uniformly contracting on \(E_\dv\);
    \item uniformly expanding on \(E_\dw\);
    \item an orientation-preserving isometry on \(E_\dz\).
\end{itemize}
Such automorphisms exist, since \(M\) is a circle bundle over the torus, these partially hyperbolic automorphisms can be constructed as circle extensions of hyperbolic toral automorphisms.
(See also \cite{RRU08,Hammerlindl13} for general information concerning partially hyperbolic automorphisms.)

Let \(\nu\) denote the probability measure on $M$ which is inherited from the Haar measure on \(\bH\).
Observe that \(\nu\) is a \(\Phi\)-invariant measure.
Since \(\Phi\) is an isometry in the \(\dz\) direction it is convenient to define, for all \(N\in \bZ\) and \(r\in \bN_0\),
\begin{equation}
    \label{eq:def-CN}
    \cC_N^r = \{h \in \cC^r(M) : \dz h = 2\pi i NK h\}, \text{ and } \ \cC_N^\infty = \cap_{r \in \bN_0} \cC_N^r.
\end{equation}
This is Fourier decomposition in the \(\dz\) direction.
Observe that \(h\mapsto h\circ \Phi\) leaves \(\cC_N^\infty\) invariant.
Consequently, in this case, we can study the action of \(\Phi\) on \(\cC_N^\infty\) for each \(N\) rather than studying it directly on \(\cC^\infty(M)\). 
This helps greatly here but is often not possible for general partially hyperbolic systems because the centre direction can't be guaranteed to be of sufficiently good regularity.

As a central component of this work, in Section~\ref{sec:anisotropic} we define a family of Banach spaces, \(\{\cB_{N}^{p,q}\}_{p,q \in \bN}\).
These are geometric-type anisotropic spaces and, by construction, \(\cC_{N}^\infty\) is a dense subset of \(\cB_{N}^{p,q}\).
Let \(\cD_N^1\) denote the space of distributions on \(\cC^\infty\) which are supported on \(\cC^1_{-N}\).

We are interested in studying the operator \(\cL : \cC_{N}^\infty \to \cC_{N}^\infty\), defined as \(h \mapsto h \circ \Phi\).
As a first step we prove that, for each \(p\), \(q\), the operator extends to a bounded linear operator on \(\cB_{N}^{p,q}\) and the essential spectral radius is not greater than \(\lambda^{-\min{\{p,q\}}}\).
This and other estimates related to the norms and operator \(\cL\) are the content of Section~\ref{sec:norm-est}.
Then, in Section~\ref{sec:self-similar} we take advantage of the special nature of the current setting in order to show a self-similar structure of the point spectrum. 
The combination of these results (Lemma~\ref{lem:spectrum}, Lemma~\ref{lem:iterated-spec}) gives us the following.

\begin{theorem}
    \label{thm:spectrum}
    Let \(M\) be a Heisenberg nilmanifold with \(K\) associated to the lattice~\eqref{eq:gamma-K}.
    Let \(\Phi : M \to M\) be a partially hyperbolic automorphism~\eqref{eq:part-hyp} with \(\lambda > 1\).
    Let \(p \in \bN_0\), \(q \in \bN\), \(N \neq 0\).
    Then
    \begin{enumerate} 
        \item The operator \(\cL : \cC_{N}^\infty \to \cC_{N}^\infty\), extends to a bounded linear operator on \(\cB_{N}^{p,q}\);
        \item The operator\footnote{The extension of the operator is denoted by the same symbol.} \(\cL : \cB_{N}^{p,q} \to \cB_{N}^{p,q}\), \(N\neq 0\) is quasi-compact with essential spectral radius not greater than \(\lambda^{-\min{\{p,q\}}}\), moreover \(\spec{{\left.\cL\right|}_{\ker_N(\dv)}}\) consists of a finite set of eigenvalues.       
        \item The spectrum of \(\cL : \cB_{N}^{p,q} \to \cB_{N}^{p,q}\), \(N\neq 0\), restricted to \(\{z\in \bC: \lambda^{-\min{\{p,q\}}} < \abs{z}\}\) is equal to 
        \[
          \{z\in \bC : \exists k \in \bN_{0}, k < \min(p,q), \lambda^k z \in \spec{{\left.\cL\right|}_{\ker_N(\dv)}} \}.
        \]
    \end{enumerate}
\end{theorem}

In other words, we obtain a full description of the spectrum if we know the spectrum of \({\left.\cL\right|}_{\ker_N(\dv)}\). 

In Section~\ref{sec:anisotropic} we introduce the anisotropic norms and study their basic properties.
In Section~\ref{sec:norm-est} we estimate the essential spectral radius of the composition operator. 
In Section~\ref{sec:self-similar} we show that the rest of the spectrum is a scaled version of the peripheral spectrum and so complete the proof of the above theorem. 

\begin{remark}
    \label{rem:N-is-zero}
    The resonance spectrum of \(\Phi\) on \(\cC^{\infty}_{0}\) is equal to \(\{1\}\) because, in the case \(N=0\), the system reduces to the study of a toral automorphism and the resonance spectrum can be shown by considering a Fourier series decomposition on the torus.
    The spectrum of \(\cL : \cB_{0}^{p,q} \to \cB_{0}^{p,q}\) is the same. 
    Some of the arguments in the paper still hold in the case \(N=0\) however the peripheral spectrum is different in this case and the argument (see Section~\ref{sec:self-similar}) which deduces the inner part of the spectrum from the peripheral spectrum fails.
\end{remark}

In Section~\ref{sec:inv-dist} we study the exterior part of the spectrum. 
In this present setting of nil-automorphisms, the spectrum of \({\left.\cL\right|}_{\ker_N(\dv)}\) corresponds to the ``invariant distributions'' which were identified by Flaminio \& Forni~\cite{FF06}.
They showed that the spectrum of the operator restricted to the kernel of \(\dv\) has multiplicity \(K \abs{N}\) and consists of eigenvalues of absolute value \(\lambda^{-\frac{1}{2}}\).
However they worked with spaces of distributions whereas we work with anisotropic spaces. 
As such their work simply tells us that the spectrum of \({\left.\cL\right|}_{\ker_N(\dv)}\) is contained within the given set of multiplicity \(K \abs{N}\), consisting of eigenvalues of absolute value \(\lambda^{-\frac{1}{2}}\).
This already, together with the above result, allows a substantial description of the spectrum.
To complete the story we would like to guarantee that the invariant distributions identified by Flaminio \& Forni~\cite{FF06} are present in the anisotropic spaces \(\cB_{N}^{p,q}\).
I.e., we need the following assumption.
For any \(r\in\bN_\infty\) denote by \(\cD_{N}^{r}\) the elements of the dual space \({\cC^{r}(M)}'\) which have support in \(\cC_{-N}^{\infty}\).

\begin{assumption}
    \label{ass:inv-dist-in-Bpq}
    Suppose that \(\cD \in \cD_N^1\) and that \(V_{*}\cD = 0\).
    Then, for every \(p\in\bN_0\) there exists \(h \in \cB_{N}^{p,1}\) such that \(\cD = \iota h\).
\end{assumption}

In other words, this assumption says that if a distribution is invariant in the \(\dv\) direction then it is contained within our spaces. 
Our anisotropic spaces require good regularity in the \(\dv\) direction and allow distribution-like behaviour in the \(\dw\) direction and so it is convincing.
Unfortunately we don't include a proof of this claim. 
Note that this assumption is not used in any of the sections prior to Section~\ref{sec:inv-dist}.
In Section~\ref{sec:inv-dist} we use the assumption and hence upgrade the above results to the following.

\begin{theorem}
    \label{thm:full-spectrum}
    Let \(M\) be a Heisenberg nilmanifold with \(K\) associated to the lattice~\eqref{eq:gamma-K}.
    Let \(\Phi : M \to M\) be a partially hyperbolic automorphism~\eqref{eq:part-hyp} with \(\lambda > 1\).
    Suppose that Assumption~\ref{ass:inv-dist-in-Bpq} holds.
    For each \(N \neq 0\), there exist a set of unit complex numbers \({\{\mu_j\}}_{j=1}^{K\abs{N}}\) and a family of Banach spaces \(\{\cB_{N}^{p,q}\}_{p,q \in \bN}\) such that for all \(p \in \bN_0\), \(q \in \bN\),
    \begin{enumerate} 
        \item The operator \(\cL : \cC_{N}^\infty \to \cC_{N}^\infty\), defined as \(h \mapsto h \circ \Phi\), extends to a bounded linear operator on \(\cB_{N}^{p,q}\);
        \item The spectrum\footnote{The extension of the operator is denoted by the same symbol.} of \(\cL : \cB_{N}^{p,q} \to \cB_{N}^{p,q}\), outside of \(\{\abs{z} \leq \lambda^{-\min{\{p,q\}}}\}\), is equal to
        \[
            \big\{\lambda^{-(\frac{1}{2} + n)}\mu_{j} : 1 \leq j \leq K \abs{N}, n\in \{0,\ldots,\min{\{p,q\}}-1\} \big\},
        \]
        with values repeated according to the multiplicity of the eigenvalues. 
    \end{enumerate}
\end{theorem}

The above spectral result implies the analogous resonance result, as per the reference~\cite{FGL19} yet without Jordan blocks in this case.

\begin{theorem}
    \label{thm:resonances}
    Let \(M\) be a Heisenberg nilmanifold with \(K\) associated to the lattice~\eqref{eq:gamma-K}.
    Let \(\Phi : M \to M\) be a partially hyperbolic automorphism~\eqref{eq:part-hyp} with \(\lambda > 1\).
    For each \(N \neq 0\) there exists a set of unit complex numbers \({\{\mu_j\}}_{j=1}^{K\abs{N}}\) such that, setting
    \[
        \Xi_N = \big\{\lambda^{-(\frac{1}{2} + n)}\mu_{j} : 1 \leq j \leq K \abs{N}, n\in \bN_{0} \big\},
    \]
    then, for any \(g,h \in \cC^{\infty}_{N}\) and for any \(\epsilon > 0\), there is an asymptotic expansion
    \[
        \int g \cdot h\circ \Phi^n \ d\nu = \sum_{\substack{\xi\in\Xi_N \\ \abs{\xi}\geq \epsilon}} \xi^n  c_{j,k}(g,h) + \littleo{\epsilon^n}
    \]
    where \(c_{j,k}(g,h)\) are non-zero, finite rank, bi-linear functions of $g$ and $h$.
\end{theorem}

See the work of Faure, Gouëzel \& Lanneau~\cite{FGL19} and references within for further discussion on the topic of resonances in dynamical systems.

The present work is closely related to the work of Flaminio \& Forni~\cite{FF06} and to the work of Faure \& Tsujii~\cite{FT15} (see also \cite{Faure07}).

In the first mentioned work, the authors study a cocycle rather than just the periodic case as here.
However, due to their choice of Hilbert space (standard Sobolev space, not anisotropic) they are unable to deduce anything about the spectrum of the operator and consequently take an indirect route in order to deduce the resonance spectrum.

The latter of the two mentioned works is a study of the ``prequantum transfer operator'' of Anosov diffeomorphism on circle bundles of symplectic manifolds.
Such a setting includes the setting of the present work but also allows for non-affine systems which are close to affine. 
Their results describe bands of spectrum which correspond to the circles of eigenvalues deduced in the present work.
Anisotropic Banach spaces are also used there in order to achieve these results.
However in that case the spaces are much more complex and the study of the transfer operator is highly technical and involved.
Indeed an entire 236 page book is devoted to the proof of the spectral result.
In this present work, we rely on purely geometric constructions for anisotropic space \(\cB^{p,q}_{N}\) (see Definition \ref{def:Bpq}) on the nilmanifold, which, although at present limited to the affine case, are much simpler than those previously used. 

Another particular advantage of the anisotropic Banach spaces (compared to \cite{FF06}) is utilized in Section~\ref{sec:self-similar} where the full spectrum is obtained from the peripheral spectrum. 
This takes the place of an argument concerning the formal inverse of a given operator of the cohomological equations in work of Flaminio \& Forni (see \cite[A.3]{FF06}). 
Potentially, the construction of anisotropic space in this paper is an initial step toward defining transfer cocycle to fully extend results to the non-algebraic systems (see \cite[\S 6]{Forni20}).

%%%%%%%%%%%%%%%%%%%%%%
\section{Anisotropic spaces}
\label{sec:anisotropic}
%%%%%%%%%%%%%%%%%%%%%%

In this section we define norms on \(\cC_N^\infty\) and then define Banach spaces by completion with respect to the norms.
We also explore various basic properties of these norms.
The norms used are similar to the geometric anisotropic norms often used for hyperbolic systems \cite{GL19,GL06,GL08} and those used by Faure, Gouëzel \& Lanneau~\cite{FGL19} for pseudo-Anosov maps.

Let \(\Phi:M \to M\) and \(N\in \bZ\), as introduced in Section~\ref{sec:results}, be fixed for the remainder of this section.
Here we permit the case \(N=0\).
For each \(t\in \bR\) let
\[
    \varphi_{t}^{\dw} : M \to M
\]
be the flow generated by $\dw$, defined as \((x,y,z) \mapsto (x,y,z) * \exp(t\dw)\). 
In other words, the flow corresponds to sliding along the \(\dw\)-foliation. 
The flows \((\varphi_{t}^{\dv})_{t\in\bR}\) and \((\varphi_{t}^{\dz})_{t\in\bR}\) are defined similarly. 

Observe that the partial hyperbolicity \eqref{eq:part-hyp} implies that, for each \(t\in \bR\)  
\begin{equation}
    \label{eq:renorm}
    \Phi\circ\varphi_{t}^{\dw}  =  \varphi_{\lambda t}^{\dw}\circ \Phi.
\end{equation}
This is the way in which the action of \(\Phi\) can be used to renormalize the flow along \(\dw\).

As mentioned earlier, \(M\) is a fiber bundle over a torus. 
The \emph{abelianization} of the Heisenberg group $\bH$, defined by $\bH^{ab} = \bH / [\bH,\bH]$, is isomorphic to $\bR^2$.
The abelianization of the lattice is defined as $\Gamma^{ab} = \Gamma / [\Gamma,\Gamma].$ 
Thus there is a natural projection 
\begin{equation}
    \label{eq:project}
    \pi: M \rightarrow  \bH^{ab}/\Gamma^{ab} \simeq \mathbb{T}^2,
\end{equation}
where $\pi(x,y,z) = (x,y)$. (See e.g., \cite{AGH63} for these and related details.)
It is known that the following are equivalent:
\begin{enumerate}
    \item The flow $(\varphi_t^\dw)_{t\in\bR}$ on $M$ is minimal;
    \item The projected flow $(\pi \circ \varphi_t^\dw)_{t\in\bR}$ on $\bT^{2}$ is an irrational linear flow.
\end{enumerate}
Since \(\Phi\) is partially hyperbolic, the projection onto $\bT^{2}$ is a hyperbolic automorphism.
%That it is a hyperbolic toral automorphism 
This implies that the vector field \(\dw\) tangent to invariant unstable manifolds has irrational slope on $\bT^{2}$,  consequently $(\varphi_t^\dw)_{t\in\bR}$ is minimal.
In particular each leaf of the \(\dw\)-foliation is dense. 
The same result is also true for the leaves of \(\dv\)-foliation.

\begin{remark}
    \label{rem:projection}
    As a fundamental domain for \(M\) we can choose \([0,1) \times [0,1) \times [0,1/K)\).
    The identifications of the edges are\footnote{In order to obtain the identifications we observe that \((1,0,0) * (0,y,z) = (1,y,z+y)\), \((0,1,0) * (x,0,z) = (x,1,z)\), \((0,0,1/K) * (x,y,0) = (x,y,1/K)\).}  
    \[
        (0,y,z) \sim (1,y,z+y/K ),\ (x,0,z) \sim (x,1,z), (x,y,0) \sim (x,y,1/K).
    \]
    I.e., two faces of the cube are identified by standard translation and the third is identified with a twist (cf. \cite{Shi14b}).
    Since functions in \(\cC^\infty_N\) have a specific behaviour in the third coordinate, \(\cC^\infty_N\) can be identified with a space of functions defined on the unit square with the additional requirement of how the function and derivatives match up at the boundary according to the identification of the edges.
    There cannot exist functions in \(\cC^\infty_N\) supported in a neighbourhood of a point.
    However this justifies the existence of functions in \(\cC^\infty_N\) which are supported on the neighbourhood of \(\pi^{-1} (x)\) for any \(x \in \mathbb{T}^2\).% (here \(\pi\) is the canonical projection of the fibre bundle~\eqref{eq:project}).
\end{remark}

We fix \(\delta > 0\), once and for all, much smaller than the diameter of \(M\).  
There exists a covering of \(\pi (M) = \bT^2\) consisting of sets of diameter not greater than \(\delta\) and, subordinated to this, there exists a \(\cC^\infty\) partition of unity.
Taking the pull back under \(\pi\) leads to a partition of unity of \(M\), i.e., a set of functions \(\{\rho_k\}_k\) such that, for each \(k\),
\begin{equation}
    \label{eq:part-unity}
    \rho_k \in \cC^\infty(M), \quad \sum_{k} \rho_k = 1.   
\end{equation}
%(if  \(\{\tilde{\rho}_k\}_k\) is a partition of unity of \(\bT^2\) then we take \(\rho_k = \tilde{\rho}_k \circ \pi\)).
The projection of the support of \(\rho_k\) will be contained within a \(\delta\)-ball however \(\rho_k\) will be constant along each fibre, i.e., constant along a circle \(\pi^{-1}(x)\) for each \(x \in \bT^2\).
In particular this means that if \(h \in \cC^\infty_N\) then \(\rho_k \cdot h \in \cC^\infty_N\).

For notational convenience let \(I_\delta = (-\delta,\delta)\) and let \(\cS(I_{\delta})\) denote the set of \(\cC^\infty\) functions with support compactly contained in the interval \(I_{\delta}\).
For \(h\in \cC_N^\infty\), \(\eta \in \cS(I_{\delta})\) and \(m\in M\) let,
\begin{equation}
    \label{eq:def-eta}
    \ell_{\eta,m}(h)
    = \int_{-\delta}^{\delta} \eta(t) \cdot h\circ \varphi_{t}^{\dw}(m) \ dt.
\end{equation}
For \(p\in\bN_0\), \(q\in\bN\) we define a norm on \(\cC_N^{\infty}\),
\[
    \norm{h}_{p,q}= \sup
    \left\{ \abs{\ell_{\eta,m}(\dv^j h)} : 0 \leq j\leq p, m\in M, \eta \in \cS(I_{\delta}),  \norm{\eta}_{\cC^{q}}\leq 1\right\}.
\]

\begin{remark}
    We use the \(\cC^r\) norm defined as \(\norm{f}_{\cC^r} = \sup_{k\leq r} 2^{r-k} \abs{\smash{f^{(k)}}}_{\infty} \).
    This has the convenient property that \(\norm{f g}_{\cC^r} \leq \norm{f}_{\cC^r} \norm{g}_{\cC^r}\).
\end{remark}

\begin{definition}
    \label{def:Bpq}
    The Banach space \(\cB_{N}^{p,q}\) is defined as the completion of \(\cC_N^\infty\) with respect to the norm $\norm{\cdot}_{p,q}.$  
\end{definition}

By definition, whenever \(p\leq p'\) and \(q'\leq q\), there is a continuous inclusion 
\begin{equation}
    \label{eq:continuous}
    \cB_{N}^{p',q'}(M) \subseteq \cB_{N}^{p,q}(M).
\end{equation}
These norms are \(\cC^r\)-like in the \(\dv\)-direction and distribution-like in the \(\dw\)-direction.

The following clarifies the dependence of the norm on the choice of \(\delta > 0\).

\begin{lemma}
    \label{lem:norm-scale}
    Given \(q\in \bN\), there exists \(C>0\) such that, for all \(\eta \in \cC^{\infty}(\bR)\) supported in some interval \(A \subset \bR\) of length \(\abs{A} \geq 2\delta\) and for all \(h\in \cC^{\infty}_{N}\),
    \[
        \abs{\int \eta(t) \cdot  h\circ \varphi_{t}^{\dw}(m) \ dt}
        \leq C \abs{A} \norm{\eta}_{\cC^q} \abs{h}_{0,q}.
    \]
\end{lemma}

\begin{proof}
    Using a $\delta$-periodic partition of unity\footnote{I.e., fix a smooth function \(\rho : \bR \to [0,1]\) such that the support of \(\rho\) is contained within \([-\delta,\delta]\) and, for every \(x \in \bR\), \(\sum_{j\in \bZ} \rho_j(x) = 1\) where we have defined \(\rho_j(x) = \rho(x-j\delta)\).} on $\bR$, there exist $C_1,C_2>0$ such that, for any \(\eta\) we can write
    \(\eta = \sum_{i= 1}^{\lfloor C_1 \abs{A}/\delta \rfloor} \tilde{\eta}_{i}\)
    where each \(\tilde{\eta}_{i}\) is supported on an interval smaller than \(2 \delta\) and \(  \norm{\tilde{\eta}_{i}}_{\cC^{q}}\leq C_2  \norm{\eta}_{\cC^q}\).
    Consequently
    \begin{align*}
        \abs{\int \eta(t) \cdot  h\circ \varphi_{t}^{\dw}(m) \ dt}
         & \leq \sum_{i= 1}^{\lfloor C_1 \abs{A}/\delta \rfloor}  \abs{\int \tilde{\eta}_{i}(t) \cdot  h\circ \varphi_{t}^{\dw}(m) \ dt} \\
         & \leq C_1 \abs{A} \delta^{-1} C_2 \norm{\eta}_{\cC^q} \abs{h}_{0,q}.  \qedhere
    \end{align*}
\end{proof}

In the following we see that the linear functionals used in the definition of the norm have a certain continuous dependence on where they are centred.

\begin{lemma}
    \label{lem:slide}
    Suppose that \(m \in M\), \(q\in\bN\), \(0<\epsilon<\delta\), and \(\eta \in \cS(I_{\delta})\), \(\norm{\eta}_{\cC^q} \leq 1\).
    Let \(\tilde{m} = \varphi_c^\dz \circ \varphi_b^\dw \circ \varphi_a^\dv (m)\) and \(\tilde{\eta}(t) = e^{- 2\pi i NK (c + at)} \eta(t)\).
    Then there exists \(C>0\) such that for all \(h \in  \cC^\infty_N\) and any \(a,b,c \in \bR\) with \(\abs{a}, \abs{b} \leq \epsilon\), %the following holds.
    \[
        \abs{\ell_{\eta,m}(h) - \ell_{\tilde{\eta},\tilde{m}}(h)} \leq C \epsilon \norm{h}_{1,q-1}.
    \]
\end{lemma}

\begin{proof}
    Rearranging the integral defining the linear functional,
    \begin{equation}
        \label{eq:slide1}
        \begin{aligned}
            \ell_{\tilde{\eta},\tilde{m}}(h)
             & = \int \tilde{\eta}(t) \cdot h \circ \varphi_{t+b}^{\dw} \circ \varphi_a^\dv \circ \varphi_c^\dz (m) \ dt                                     \\
             & = \int \tilde{\eta}(t) \cdot h \circ \varphi_{t}^{\dw} \circ \varphi_a^\dv \circ \varphi_c^\dz (m) \ dt                                       \\
             & \ \ + \int \left[\tilde{\eta}(t-b) - \tilde{\eta}(t)\right] \cdot h \circ \varphi_{t}^{\dw} \circ \varphi_a^\dv \circ \varphi_c^\dz (m) \ dt.
        \end{aligned}
    \end{equation}
    Continuing with the first part of the last term, recalling that \(\varphi_{t}^{\dw} \circ \varphi_a^\dv = \varphi_a^\dv  \circ \varphi_{t}^{\dw} \circ \varphi_{at}^\dz\),
    \begin{multline}
        \label{eq:slide2}
            \int \tilde{\eta}(t) \cdot h \circ \varphi_{t}^{\dw} \circ \varphi_a^\dv \circ \varphi_c^\dz (m) \ dt
              = \int \tilde{\eta}(t) \cdot h \circ \varphi_a^\dv  \circ \varphi_{t}^{\dw} \circ \varphi_{c+at}^\dz (m) \ dt                               \\
             = \int \tilde{\eta}(t) \cdot h \circ \varphi_{t}^{\dw} \circ \varphi_{c+at}^\dz (m) \ dt                                                    
              \ \ + \int \int_{0}^{a} \tilde{\eta}(t) \cdot \dv h \circ \varphi_{s}^{\dv} \circ \varphi_{t}^{\dw} \circ \varphi_{c+at}^\dz (m) \ ds \ dt.
    \end{multline}
    Since \(h \in \cC^{\infty}_{N}\) and hence \(h \circ \varphi_{t}^{\dw}  \in \cC^{\infty}_{N}\), by definition of \(\tilde{\eta}\),
    \begin{equation}
        \label{eq:slide3}
        \begin{aligned}
            \int \tilde{\eta}(t) \cdot h \circ \varphi_{t}^{\dw} \circ \varphi_{c+at}^\dz (m) \ dt
             & = \int e^{2\pi i NK (c + at)} \tilde{\eta}(t) \cdot h \circ \varphi_{t}^{\dw} (m) \ dt \\
             & = \int \eta(t) \cdot h \circ \varphi_{t}^{\dw} (m) \ dt                               \\
             & = \ell_{\eta,m}(h).
        \end{aligned}
    \end{equation}
    Combining the above equalities \eqref{eq:slide1}, \eqref{eq:slide2} and \eqref{eq:slide3}, we have shown that
    \begin{equation}
        \label{eq:slide4}
        \begin{aligned}
            \abs{\ell_{\eta,m}(h) - \ell_{\tilde{\eta},\tilde{m}}(h)}
             & \leq \abs{\int \left[\tilde{\eta}(t-b) - \tilde{\eta}(t)\right] \cdot h \circ \varphi_{t}^{\dw} \circ \varphi_a^\dv \circ \varphi_c^\dz (m) \ dt} \\
             & \ \ + \abs{\int \int_{0}^{a} \tilde{\eta}(t) \cdot \dv h \circ \varphi_{s}^{\dv} \circ \varphi_{t}^{\dw} \circ \varphi_{c+at}^\dz (m) \ ds \ dt}.
        \end{aligned}
    \end{equation}
    For convenience when bounding the first of these two terms, let \(\zeta(t) = \tilde{\eta}(t) - \tilde{\eta}(t-b)\).
    Observe that the support of \(\zeta\) is contained within an interval of length \(2(\delta + \epsilon) \leq 4\delta\).
    Moreover, for each $n\in \bN$,  
    \[\zeta^{(n)}(t) = \tilde{\eta}^{(n)}(t) - \tilde{\eta}^{(n)}(t-b) = \int_{t-b}^{t} \tilde{\eta}^{(n+1)}(s) \ ds,\] and so \(\norm{\zeta}_{\cC^{q-1}} \leq b \norm{\tilde{\eta}}_{\cC^{q}}\).
    By assumption \(\norm{\eta}_{\cC^q} \leq 1\) and so there is an upper bound for \(\norm{\tilde{\eta}}_{\cC^{q}}\), uniform in \(\eta\), depending only on \(q\).
    Consequently, using also Lemma~\ref{lem:norm-scale}, there exists \(C_1>0\) such that,
    \begin{equation}
        \label{eq:slide5}
        \abs{\int \left[\tilde{\eta}(t-b) - \tilde{\eta}(t)\right] \cdot h \circ \varphi_{t}^{\dw} \circ \varphi_a^\dv \circ \varphi_c^\dz (m) \ dt} 
        \leq \epsilon C_1 \norm{h}_{0,q-1}.
    \end{equation}
    For the final term  \eqref{eq:slide4} which remains to estimate, 
    \begin{multline*}
        \int \left( \int_{0}^{a} \tilde{\eta}(t) \cdot \dv h \circ \varphi_{s}^{\dv} \circ \varphi_{t}^{\dw} \circ \varphi_{c+at}^\dz (m) \ ds \right) \ dt \\
        = \int_{0}^{a} \left( \int \tilde{\eta}(t) e^{2\pi i (c + at + st)N} \cdot \dv h  \circ \varphi_{t}^{\dw} \circ \varphi_{s}^{\dv} (m) \ dt \right) \ ds \\
        \leq a \norm{h}_{1,{q}} \norm{\smash{t \mapsto  \tilde{\eta}(t) e^{2\pi i (c + at + st)N} }}_{\cC^{q}}.
    \end{multline*}
    Consequently, there exists a constant \(C_2>0\), depending only on \(N\) and \(q\) such that this term is bounded above by
    \begin{equation}
        \label{eq:slide6}
        \epsilon C_2 \norm{h}_{1,q}.
    \end{equation}
    Combining the two estimates we have obtained, \eqref{eq:slide5} and \eqref{eq:slide6}, whilst noting that \(\norm{h}_{0,q-1} \leq \norm{h}_{1,q-1}\) and \(\norm{h}_{1,q} \leq \norm{h}_{1,q-1}\), completes the proof of the lemma.
\end{proof}

The following lemma means that one can identify \(\cB_{N}^{p,q}\) with a space of distributions. 
Recall that \(\cD_{N}^{r}\) denotes the elements of the dual space \({\cC^{r}(M)}'\) which have support in \(\cC_{-N}^{\infty}\).

\begin{lemma}
    \label{lem:distributions}
    Let \(p \in \bN_0\), \(q \in \bN\).
    The canonical inclusion map \(\iota:\cC^\infty_{N} \rightarrow \cD_{N}^{\infty}\) given by \(\langle \iota(h), g \rangle =  \int_{M} h \cdot g \) extends to a map \(\cB_{N}^{p,q} \to \cD_{N}^{q}\) which is continuous and injective.
\end{lemma}

\begin{proof}
    Using the previously introduced special partition of unity of \(M\) \eqref{eq:part-unity} and decomposing the integral into orbit segments of the flow \(\mathcal{O}_{m,\delta} = \left\{\varphi_t^\dw(m) : t\in (-\delta,\delta) \right\}\), one can show that for all \(h\in \cC_N^\infty\),
    \[
        \abs{\int_{M} h \cdot g} \lesssim \norm{h}_{p,q} \norm{g}_{\cC^p}.
    \]
    Here we have used that, for the leaf-wise integrals, \(\int_{\mathcal{O}_{m,\delta}}  h \cdot g = \int_{-\delta}^{\delta} \eta(t) \cdot h \circ \varphi^{\dw}_{t}(m)\ dt\) if we choose \(\eta(t) = g \circ \varphi^{\dw}_{t}(m)\).
    Considering the completion, this shows that any \(h \in \cB_{N}^{p,q}\) gives a distribution on \(M\) of order at most \(p\).
    
    That \(\iota\) is injective follows in the same way as the argument of Gouëzel \& Liverani~\cite[Proposition 4.1]{GL06}.
    We start by taking \(h \in \cB_{N}^{p,q}\), \(h\neq 0\).
    Consequently there exist some \(m\in M\) and \(\eta \in \cS(I_\delta)\) such that \(\int_{-\delta}^{\delta} \eta(t) \cdot h\circ\varphi_t^{\dw}(m)\ dt\) is non-zero.
    We then use this \(\eta\) to construct a \(g\in \cC^{\infty}_{-N}\) such that \(\langle \iota h, g\rangle \neq 0\).
    This can be done, as discussed in Remark~\ref{rem:projection}, in such a way that \(g\) is supported on a neighbourhood of \(\{\varphi^{\dz}_t(m): t\in \bR\}\). 
    This then implies that \(\iota (h) \in \cD_{N}^{q}\) is non-zero as required.
\end{proof}

The following compact embedding result and its argument are very similar to the one appearing in other works using geometric anisotropic space (e.g., \cite[\S5]{GL06} and \cite[\S2.2]{FGL19}).

\begin{lemma}
    \label{lem:compact-embed}
    Let \(p,p'\in \bN_0\), \(q,q'\in\bN\), \(p < p'\), and \(q' < q\).
    Then, the inclusion \(\cB_{N}^{p',q'} \subset \cB_{N}^{p,q}\) is compact.
\end{lemma}

\begin{proof}
    Since there is a continuous embedding in \eqref{eq:continuous}, it suffices to prove the result for the case \(p' = p+1\), \(q'=q-1\).
    By density it suffices to work with \(h\in \cC^{\infty}_{N}\).
    We will show that, for each sufficiently small \(\epsilon>0\) there exists a finite set \({\{\rho_k\}}_{k}\) of linear functionals on \(\cB_{N}^{p+1,q-1}\) such that, for any \(h\in \cC^{\infty}_{N}\),
    \begin{equation}
        \label{eq:compact-bound}
        \norm{h}_{p,q} \leq \max_{k}\abs{\rho_k(h)} + \epsilon \norm{h}_{p+1,q-1}.
    \end{equation}
    By a diagonal argument, this implies the claimed compactness (see, e.g., \cite[Proof of Prop.\,2.8]{FGL19}).

    Fix \(0<\epsilon < \delta\).
    Let \({\{m_k\}}_{k}\) denote a finite set of points in \(M\) such that every point of \(M\) is \(\epsilon\)-close to at least one of the \(m_k\).
    According to Lemma~\ref{lem:slide}, for all \(m \in M\) there exists an index \(k\) such that, for all \(h\in \cC^\infty_N\), \(\eta \in \cS(I_\delta)\), \(\norm{\eta}_{\cC^q} \leq 1\),
    \begin{equation}
        \label{eq:compact1}
        \abs{\ell_{\eta,m}(h) - \ell_{\tilde{\eta},m_k}(h)} \lesssim \epsilon \norm{h}_{1,q-1},
    \end{equation}
    where \(\tilde{\eta}(t) = e^{- 2\pi i N (c + at)} \eta(t)\).

    Let \(C = \sup_{0 \leq a \leq \delta}\norm{\smash{t\mapsto e^{- 2\pi i N at}}}_{\cC^q}\).
    Using the compactness of \(\cC^{q}\) in \(\cC^{q-1}\), we choose a finite set \({\{\zeta_i\}}_{i}\) of functions in \(\cS(I_\delta)\) such that, for every \(\eta \in \cS(I_\delta)\), \(\norm{\eta}_{\cC^q} \leq C\), there exists some \(\zeta_i\) in the set such that \(\norm{\eta - \zeta_i}_{\cC^{q-1}}\leq \epsilon\).
    Since
    \[
        \begin{aligned}
            \abs{\ell_{\eta,m}(h) - \ell_{\zeta_i,m}(h)}
             & = \abs{\int \left[ \eta(t) - \zeta_i(t) \right] \cdot h\circ \varphi_{t}^{\dw}(m) \ dt  } \\
             & \leq  \norm{\eta - \zeta_i}_{\cC^{q-1}} \norm{h}_{0,q-1}
        \end{aligned}
    \]
    we know that, for every \(\eta \in \cS(I_\delta)\), \(\norm{\eta}_{\cC^q} \leq C\) there exists an index \(i\) such that, for every \(m \in M\),
    \begin{equation}
        \label{eq:compact2}
        \abs{\ell_{\tilde{\eta},m}(h) - \ell_{\zeta_i,m}(h)} \lesssim \epsilon \norm{h}_{0,q-1}.
    \end{equation}

    Combining the above estimates, \eqref{eq:compact1}, \eqref{eq:compact2}, we have shown that, for every \(m \in M\), \(\eta \in \cS(I_\delta)\), and \(\norm{\eta}_{\cC^q} \leq C\), there exist indexes \(i,k\) such that, for all  $0\leq j \leq p$,
    \[
        \abs{\ell_{\eta,m}(\dv^j h) - \ell_{\zeta_i,m_k}(\dv^j h)} \lesssim \epsilon \norm{\dv^j h}_{1,q-1} \leq  \epsilon \norm{h}_{p+1,q-1}.
    \]
    As such, the set of linear functionals defined as \(h \mapsto \ell_{\zeta_i,m_k}(\dv^j h)\) and indexed by \(i,j,k\), satisfy the required property~\eqref{eq:compact-bound}.
\end{proof}

\begin{lemma}
    \label{lem:cont-X}
    For any \(p,q \in \bN\), %we have \(\norm{\dv h}_{p-1,q} \leq \norm{h}_{p,q}\). I.e,
    \(\dv:\cC^{\infty}_N \rightarrow \cC^{\infty}_N\) extends to a continuous operator \(\cB_{N}^{p,q} \to \cB_{N}^{p-1,q}.\)
    %The operator \(\dv: \cB_{N}^{p,q}(M) \to \cB_{N}^{p-1,q}(M)\) is continuous.%
    %\footnote{In the sense that \(\dv\) extends to a continuous operator on these spaces and, abusing notation, we use the same symbol for the extension.}
\end{lemma}

\begin{proof}
    Let \(h \in \cC^{\infty}_N\).
    Since \([\dv,\dz]=0\), we know that \(\dv h \in \cC^{\infty}_N\).
    Further observe that
    \[
        \begin{aligned}
            \norm{\dv h}_{p-1,q}
             & = \sup
            \left\{
            \abs{\ell_{\eta,m}(\dv^j (\dv h))} :
            0 \leq j\leq p-1, m\in M,  \norm{\eta}_{\cC^{q}}\leq 1
            \right\}                \\
             & = \sup
            \left\{
            \abs{\ell_{\eta,m}(\dv^j  h)} :
            1 \leq j\leq p, m\in M,  \norm{\eta}_{\cC^{q}}\leq 1
            \right\}                \\
             & \leq \norm{h}_{p,q}.
        \end{aligned}
    \]
    By density the full result follows.
\end{proof}

\begin{lemma}
    \label{lem:cont-Y}
    For any \(p \in \bN_0\), \(q \in \bN\),
    % The operator \(\dw : \cB_{N}^{p,q} \to \cB_{N}^{p,q+1}\) is continuous.
    \(\dw : \cC^{\infty}_N \rightarrow \cC^{\infty}_N\) extends to a continuous operator \(\cB_{N}^{p,q} \to \cB_{N}^{p,q+1}\).

\end{lemma}

\begin{proof}
    Let $h\in \cC_N^\infty$.
    We must estimate \(\norm{\dw h}_{p,q+1}\).
    As such, let $\eta \in \cS(I_{\delta})$, $\norm{\eta}_{\cC^{q+1}} \leq 1$ and let \(1 \leq j \leq p\).
    The commutation relations of $\dv,\dw$ imply \(\dv \dw = \dw \dv + \dz\) so that \(\dv^j \dw = \dw \dv^j + j \dv^{j-1} \dz\). %(in the case \(j=0\) the term \(j \dv^{j-1} Z\) is not present).

    Rearranging and integrating by parts,
    \begin{align*}
         & \int \eta(t) \cdot (\dv^j\dw h)\circ\varphi_t^\dw(m) \ dt                                                                \\
         & \ \ = \int \eta(t)\cdot  \dw(\dv^j h)\circ\varphi_t^\dw(m)\ dt + j\int \eta(t)\cdot \dv^{j-1}\dz h\circ \varphi_t^\dw(m) \ dt  \\
         & \ \ = - \int \eta'(t) \cdot \dv^j h \circ \varphi_t^\dw(m)\ dt + j\int \eta(t)\cdot \dv^{j-1}\dz h \circ \varphi_t^\dw(m) \ dt.
    \end{align*}
    Observe that \(\norm{\eta'}_{\cC^{q}} \leq 1\).
    This all means that,
    \begin{equation*}
        \left|\int \eta(t) \cdot (\dv^j \dw h)\circ\varphi_t^\dw(m) \ dt\right| \leq \norm{\dv^j h}_{0,q} + j \norm{\dv^{j-1}(\dz h)}_{0,q+1}.
    \end{equation*}
    And so, using also the definition of $\cC_N^\infty$, we have shown that 
    \[
        \norm{\dw h}_{p,q+1} \leq (2\pi K |N|p + 1)\norm{h}_{p,q}. \qedhere
    \]
\end{proof}

\begin{lemma}
    \label{lem:ker-W}
    Let \(p \in \bN_0\), \(q \in \bN\).
    Suppose that \(N \neq 0\), \(h\in \cB^{p,q}_{N}\) and \(\dw h = 0\).
    Then \(h = 0\).
\end{lemma}

\begin{proof}
    Since the kernel of a bounded linear operator is closed, it suffices to work with \(h\in \cC_N^\infty\).
    By assumption $\ell_{\eta,m}(\dw h)=0$ for all \(\eta\in \cS(I_{\delta})\), \(m\in M\).
    This implies that \(h\) is constant along each leaf of the \(\dw\)-foliation. 
    Such leaves are dense in \(M\) (as discussed at the start of this section).
    Since \(h\) is continuous and equal to a constant function on a dense set, it is constant on \(M\).
    However, since \(N\neq 0\), this contradicts the oscillating behaviour of \(h\) in the \(\dz\)-direction.
\end{proof}

A key characteristic of the present anisotropic setting is that Lemma~\ref{lem:ker-W} doesn't hold when \(\dw\) is replaced with \(\dv\).
Indeed, as we will see in Section~\ref{sec:inv-dist}, the set \(\{h \in \cB^{p,q}_{N} : \dv h = 0\}\) is non-empty.
Moreover this set consists of the eigenspaces associated to the peripheral spectrum.

%%%%%%%%%%%%%%%%%%%%%%
\section{Composition operator and norm estimates}
\label{sec:norm-est}
%%%%%%%%%%%%%%%%%%%%%%

This section is devoted to proving several key properties of the composition operator in relation to the anisotropic norms, in particular that the composition operator is quasi-compact.
We consider the linear operator \(\cL : \cC^\infty \to \cC^\infty\), given by
\begin{equation}
    \label{eq:def-L}
    \cL : h \mapsto h \circ \Phi
\end{equation}
and which we call the \emph{composition operator}. 

\begin{remark}
    \label{rem:transfer-op}
    The above introduced operator is also known as the Koopman operator and the dual operator is commonly known as the transfer operator. 
    In the present affine and invertible setting, the invariant measure is normalised volume and the transfer operator would be, up to some constant factor, equal to \(h \mapsto h \circ \Phi^{-1}\). 
    Consequently the spectrum of this operator can be obtained by studying the operator of the present article and swapping the role of \(\dv\) and \(\dw\) in the definition of the anisotropic space.
    For a general dynamical system, the transfer operator corresponding to the measure of maximal entropy or the one corresponding to the SRB measure will include a different weight in the definition and consequently would be expected to have a different spectrum.
    However, again because of the present affine setting, in each case the invariant measure is normalised volume and the corresponding transfer operators are equal up to a scaling constant. 
    The present choice to study the composition operator matches what was done by Faure, Gouëzel and Lanneau in their work on pseudo-Anosov maps~\cite{FGL19}.
\end{remark}

\begin{remark}
    Once we fix the anisotropic space \(\cB^{p,q}_{N}\) the spectrum of \(h \mapsto h \circ \Phi\) has minimal connection to the spectrum of \(h \mapsto h \circ \Phi^{-1}\), indeed we will obtain a precise description of the former whilst the latter will fail to be a contraction on these Banach spaces.
    That the inverse fails to be a contraction can be seen by inspecting the proofs of Lemma~\ref{lem:LY-first} and Lemma~\ref{lem:LY-other}.
    The norms behave well under the operator when expansion and contraction matches the suited directions of the anisotropic space but this means the directions can't be swapped.
\end{remark}

\begin{lemma}
    \label{lem:XL-LX}
    For all \(j,k \in \bN\) and \(h\in \cC^{\infty}_{N}\), the following relations hold:
    \[
        \dv^{j}\cL^k h = \lambda^{-jk} \cL^k\left(\dv^{j}h\right),
        \quad  
        \dw^{j}\cL^k h = \lambda^{jk} \cL^k\left(\dw^{j}h\right).
    \]
\end{lemma}

\begin{proof}
    By \eqref{eq:part-hyp},
    \[
        \dv\cL^k h = \dv (h \circ \Phi^{k}) = (\dv h\circ \Phi^{k})\lambda^{-k} =  \lambda^{-k} \cL\left(\dv h\right).
    \]
    %and so \(\dv\cL^k h = \lambda^{-k} \cL^k\left(\dv h\right)\).
    Iterating this leads to the full result.
    We repeat similarly for \(\dw\).
\end{proof}

Since \(\dv\), \(\dw\) and \(\cL\) extend continuously to $\cB^{p,q}_{N}$, the above result also extends to $\cB^{p,q}_{N}$, a fact which will soon be useful.

For convenience, for all \(j\in \bN_0\), \(q\in\bN\), we define the semi-norm on \(\cC_N^{\infty}\),
\begin{equation}
    \label{eq:def-semi}
    \abs{h}_{j,q}= \sup
    \left\{ \abs{\ell_{\eta,m}(\dv^j h)} : m\in M, \eta \in \cS(I_{\delta}),  \norm{\eta}_{\cC^{q}}\leq 1\right\}.
\end{equation}
This definition has the consequence that \(\norm{h}_{p,q}= \displaystyle\max_{0\leq j \leq p} \abs{h}_{j,q}\).

\begin{lemma}
    \label{lem:LY-first}
    There exists \(C>0\) such that, for all \(k \in \bN_0\), \(q \in \bN\), and \(h\in \cC^{\infty}_{N}\),
    \[
        \snorm{\smash{\cL^k h}}_{0,q} \leq C \snorm{h}_{0,q}.
    \]
\end{lemma}

\begin{proof}
    Let \(m\in M\) and \(\eta \in \cS(I_{\delta})\) such that \(  \norm{\eta}_{\cC^{q}}\leq 1\).
    For $j =0$, we need to estimate
    \begin{equation}\label{eqn;lLk}
        \ell_{\eta,m}(\cL^k h)
        = \int_{-\delta}^{\delta} \eta(t) \cdot \cL^k h\circ \varphi_{t}^{\dw}(m) \ dt.
    \end{equation}
    In view of \eqref{eq:renorm}, we observe that \(\Phi^{k} \circ \varphi_{t}^{\dw} = \varphi_{\lambda^{k}t}^{\dw} \circ \Phi^{k}\) for any $t \in \bR$. 
    This implies 
    \[
        \cL^k h\circ \varphi_{t}^{\dw}
        = h\circ \Phi^{k} \circ \varphi_{t}^{\dw}
        = h\circ \varphi_{\lambda^{k}t}^{\dw} \circ \Phi^{k}.
    \]
    Changing variables in the integral (\(s= \lambda^k t\)),
    \[
        \ell_{\eta,m}(\cL^k h)
        = \lambda^{-k} \int_{-\lambda^k \delta}^{\lambda^k \delta} \eta(\lambda^{-k} s) \cdot h\circ \varphi_{s}^{\dw}(\Phi^{k}(m)) \ ds.
    \]
    By Lemma~\ref{lem:norm-scale},
    \[
        \ell_{\eta,m}(\cL^k h) \lesssim \lambda^{-k}\cdot(2\lambda^{k}\delta)\cdot\norm{\smash{\eta\circ \lambda^{-k}}}_{\cC^{q}}\snorm{h}_{0,q}\leq 2\delta\snorm{h}_{0,q},
    \]
    and so we have shown that \(\snorm{\smash{\cL^k h}}_{0,q} \lesssim \snorm{h}_{0,q}\).
\end{proof}

\begin{lemma}
    \label{lem:cont-op}
    Let \(p \in \bN_0\), \(q \in \bN\).
    The operator $\mathcal{L}$ extends to a continuous operator on $\cB_{N}^{p,q}$ and has spectral radius at most one. 
    Moreover the spectral radius of \(\cL  : \cB_{0}^{p,q} \to \cB_{0}^{p,q}\) is equal to \(1\).
\end{lemma}

\begin{proof}
    We observe that \(\snorm{h}_{j,q} = \snorm{\dv^{j} h}_{0,q} \).
    Hence the estimate of the above lemma implies also that \(\snorm{\cL^k h}_{j,q} \lesssim \snorm{h}_{j,q}\) for all \(j\in \bN\).
    This suffices to show that the operator \(\cL\) extends to a continuous operator on \(\cB_{N}^{p,q}\).   
    Moreover, the bound is uniform in \(k\) and so the spectral radius is at most one.
    The observation that constant functions are contained within \(\cC^{\infty}_{0}\) and are invariant under \(\cL\) implies that the spectral radius of \(\cL  : \cB_{0}^{p,q} \to \cB_{0}^{p,q}\) is equal to \(1\).
\end{proof}

In the following lemma we take advantage of the contraction in the \(\dv\)-direction in order to strengthen the previous estimates.

\begin{lemma}
    \label{lem:LY-top}
    There exists \(C>0\) such that, for all \(j, k \in \bN_0\), \(q \in \bN\), \(h\in \cC^{\infty}_{N}(M)\),
    \[
        \snorm{\smash{\cL^k h}}_{j,q} \leq C \lambda^{-jk} \snorm{h}_{j,q}.
    \]
\end{lemma}

\begin{proof}
    For $j \geq 1$, by definition of the norm and Lemma~\ref{lem:XL-LX},
    \begin{align*}
        \snorm{\smash{\cL^k h}}_{j,q}
        = \snorm{\smash{\dv^j \cL^k h}}_{0,q}
         & = \lambda^{-jk} \snorm{\smash{\cL^k(\dv^j  h)}}_{0,q} \\
         & \lesssim \lambda^{-jk} \snorm{\dv^j  h}_{0,q}
        =  \lambda^{-jk} \snorm{h}_{j,q}. \qedhere
    \end{align*}
\end{proof}

Lemma~\ref{lem:LY-top} implies that \( \norm{\cL^k h}_{p,q} \leq C  \lambda^{-pk} \norm{h}_{p,q} + \sup_{j<p}  \norm{\cL^k h}_{j,q} \).
We will proceed by estimating the second term, i.e., \(\norm{\cL^k h}_{j,q} \) for \(0\leq j \leq p-1\).
This is the content of Lemma~\ref{lem:LY-other}, using the expansion in the \(\dw\)-direction.
Before getting there, it is convenient to review the standard details related to the mollifier and smoothing which we will use in the remainder of this paper (see \cite[\S4]{EG15}).
For any $n \in \bN$, let \(\rho \in \cC^\infty(\bR^n, \bR_{+})\) be the standard mollifier supported on \(\{\abs{x}<1\}\) and \(\int_{\abs{x}<1} \rho(x) \ dx = 1\).
For \(\epsilon>0\), let \(\rho_\epsilon(x) = \epsilon^{-n} \rho(\epsilon^{-1}x)\) be $\epsilon$-sized mollifier.
The smoothed version of a function $\eta$ is defined as 
\[\eta_\epsilon(x) = \int \rho_{\epsilon}(x-y) \cdot \eta(y) \ dy.\]

\begin{lemma} 
    \label{lem:moll}
    For \(\epsilon > 0\) and $n = 1$, we obtain explicit estimates as follows:
    \begin{equation}
        \label{eq:bump-eta}
        \norm{\eta - \eta_{\epsilon}}_{\cC^{q-1}} \leq  \epsilon,
        \quad
        \norm{\eta_{\epsilon}}_{\cC^{q}} \leq 1,
        \quad
        \norm{\eta_{\epsilon}}_{\cC^{q+1}} \leq 2 \epsilon^{-1}.
    \end{equation}
\end{lemma}

\begin{proof}
    For the first inequality, we follow the proof of standard argument (see \cite[Theorem 4.1-(ii)]{EG15}).
    For any $j \geq 0$, set
    $$
        \eta_\epsilon^{(j)}(x)=\rho_\epsilon\star \eta^{(j)}(x)=\int_{-1}^{1}\rho(z)\cdot\eta^{(j)}(x-\epsilon z)dz.
    $$
    Then,
    \begin{align*}
        |\eta_\epsilon^{(j)}(x) - \eta^{(j)}(x)| & = \left|\int_{-1}^1 \rho(z)\cdot\eta^{(j)}(x-\epsilon z)dz\right|                                                      \\
                                                 & \leq \epsilon \int_{-1}^1|z|\cdot\rho(z)\cdot\left|\frac{\eta^{(j)}(x) - \eta^{(j)}(x-\epsilon z)}{\epsilon z} \right|\ dz.
    \end{align*}
    By applying mean-value theorem and by the fact that $\norm{\eta}_{\cC^{q}} \leq 1$, whenever $j\leq q-1$
    $$
        |\eta_\epsilon^{(j)}(x) - \eta^{(j)}(x)| \leq \epsilon\norm{\smash{\eta^{(j+1)}}}_{\cC^0} \int_{-1}^1|z|\cdot\rho(z)\ dz \leq \epsilon.
    $$
    For the second inequality, whenever $j\leq q$,
    $$
        \norm{\smash{\eta_\epsilon^{(j)}}}_{\cC^0} \leq \norm{\smash{\rho_\epsilon\star \eta^{(j)}}}_{\cC^0}  \leq \norm{\rho_\epsilon}_{L^1} \norm{\smash{\eta^{(j)}}}_{\cC^0} \leq 1.
    $$
    For the last inequality, integration by parts and the previous inequality give
    $$
        \norm{\smash{\eta_\epsilon^{(q+1)}}}_{\cC^0} = \norm{\smash{\rho'_\epsilon\star \eta^{(j)}}}_{\cC^0}  \leq \norm{\rho'_\epsilon}_{L^1} \norm{\smash{\eta^{(q)}}}_{\cC^0} \leq \norm{\rho'_\epsilon}_{L^1}.
    $$
    Since $\rho_\epsilon$ is compactly supported,
    $$
        \norm{\rho'_\epsilon}_{L^1} = 2 \int_{-\infty}^{0} \rho'_\epsilon(z)\ dz = 2\rho_\epsilon(0) = \frac{2\rho(0)}{\epsilon} \leq \frac{2}{\epsilon}.
    $$
    Therefore, we conclude the lemma.
\end{proof}

\begin{lemma}
    \label{lem:LY-other}
    For all \(j,k \in \bN_0\), \(q \in \bN\), and \(h\in \cC^{\infty}_{N}\),
    \[
        \snorm{\smash{\cL^k h}}_{j,q}
        \leq
        C \lambda^{-qk} \snorm{h}_{j,q} +
        C_k  \snorm{h}_{j,q+1},
    \]
    where \(C>0\) is a constant and \(C_k\) depends only on \(k\).
\end{lemma}

\begin{proof}
    By Lemma~\ref{lem:XL-LX} it suffices to prove the case for \(j=0\) because 
    \begin{equation}
        \label{eq:jVh}
        \snorm{\smash{\cL^k h}}_{j,q} = \snorm{\smash{\dv^j \cL^k h}}_{0,q} = \lambda^{-jk} \snorm{\smash{\cL^k \dv^j h}}_{0,q}.
    \end{equation}
    Let \(k\in\bN\), \(m\in M\),
    \(h\in \cC^{\infty}_{N}\)
    and \(\eta \in \cS(I_{\delta})\) such that \(  \norm{\eta}_{\cC^{q}}\leq 1\).
    We need to estimate $\ell_{\eta,m}(\cL^k h)$ (see \eqref{eqn;lLk}).
    For \(\epsilon > 0\) we fix \(\rho_{\epsilon}\) an \(\epsilon\)-sized mollifier and recall \(\eta_{\epsilon}= \eta \star \rho_{\epsilon}\).
    We write
    \[
        \ell_{\eta,m}(\cL^k h)
        = \int_{-\delta}^{\delta} \eta_{\epsilon}(t) \cdot \cL^k h\circ \varphi_{t}^{\dw}(m) \ dt
        + \int_{-\delta}^{\delta} \left(\eta - \eta_{\epsilon}\right)(t) \cdot \cL^k h\circ \varphi_{t}^{\dw}(m) \ dt.
    \]
    For the first integral, we use the estimate of Lemma~\ref{lem:norm-scale}, \ref{lem:LY-first} and \ref{lem:moll}. Then, we obtain 
    \[
        \begin{aligned}
        \abs{\int_{-\delta}^{\delta} \eta_{\epsilon}(t) \cdot \cL^k h\circ \varphi_{t}^{\dw}(m) \ dt}
        &\lesssim 2\delta \norm{\eta_{\epsilon}}_{\cC^{q+1}}  \abs{\smash{\cL^k h}}_{0,q+1}\\
        &\lesssim \epsilon^{-1} \abs{\smash{ h}}_{0,q+1}.
    \end{aligned}
    \]
    As the number $j$ of derivatives matters, combining the equation \eqref{eq:jVh}, we obtain
    $$
    \lambda^{-jk}\epsilon^{-1} \abs{\smash{ h}}_{j,q+1} \leq C_k(\epsilon)\abs{\smash{ h}}_{j,q+1}.
    $$ 
    For the second integral, we repeat the similar estimate for $\cC^{q}$-norm.
    In view of \eqref{eq:renorm} and the definition of $\cL$,
    let \(\tilde{\eta}(t) = \left(\eta - \eta_{\epsilon}\right)(\lambda^{-k}t)\). Note that $q$-th derivative of $\tilde{\eta}$ is bounded by $C\lambda^{-qk}$. Hence taking \(\epsilon = \lambda^{-qk}\), we obtain $\norm{\tilde \eta}_{\cC^q} \leq C\lambda^{-qk}$.
    Summing the above two estimates completes the proof of the lemma.
\end{proof}

\begin{lemma}
    \label{lem:lasota-yorke}
    For all \(k \in \bN_0\), \(p,q \in \bN\), and \(h\in \cB_{N}^{p,q}\),
    \[
        \norm{\smash{\mathcal{L}^k h}}_{p,q} \leq C\lambda ^{-\min{\{p,q\}} k} \norm{h}_{p,q} +  C_k\norm{h}_{p-1,q+1},
    \]
    where \(C>0\) is a constant and \(C_k\) depends only on \(k\).
\end{lemma}

\begin{proof}
    By density, it suffices to prove the inequality for \(h\in  \cC^{\infty}_{N}\).
    %Let \(p,q\in\bN\).
    Combining the estimates of Lemma~\ref{lem:LY-first} and Lemma~\ref{lem:LY-other} gives the following estimate.
    \begin{align*}
        \norm{\smash{\mathcal{L}^k h}}_{p,q}
         & \leq \snorm{\smash{\cL^k h}}_{p,q} +  \sup_{0\leq j \leq p-1}  \snorm{\smash{\cL^k h}}_{j,q} \\
         & \lesssim  \lambda^{-pk} \snorm{h}_{p,q}
        +  \lambda^{-qk} \norm{h}_{p-1,q} +
        \norm{h}_{p-1,q+1}                                                                              \\
         & \leq  \lambda^{-\min{\{p,q\}} k} \norm{h}_{p,q}
        + C_k  \norm{h}_{p-1,q+1}. \qedhere
    \end{align*}
\end{proof}

In the next lemma, we see that the essential spectral radius can be made as small as desired and so subsequently we need only work with the point spectrum. 

\begin{lemma}
    \label{lem:spectrum}
    Let \(p,q \in \bN\) and \(N\in \bZ \setminus \{0\}\).
    The operator \(\cL  : \cB_{N}^{p,q} \to \cB_{N}^{p,q}\) has essential spectral radius not greater than \(\lambda^{-\min{\{p,q\}}}\).
\end{lemma}

\begin{proof}
    The estimate on the essential spectral radius is due, as follows, to Hennion's argument \cite{Hennion93} (see also \cite{DKL21}).
    Let \(B_0 = \{h \in \cB_{N}^{p,q} : \norm{h}_{p,q} \leq 1\}\) and \(B_k = \cL^k B_0\).
    The compact embedding (Lemma~\ref{lem:compact-embed}) implies that \(B_0\) admits a finite cover by balls of \(\norm{\cdot}_{p-1,q+1}\)-diameter less than \(C_k^{-1} \lambda^{-\min{\{p,q\}}k}\). 
    The image under \(\cL^k\) of this cover is a cover of \(B_k\).
    Using the inequality (Lemma~\ref{lem:lasota-yorke}) we know that the sets of this cover have \(\norm{\cdot}_{p,q}\)-diameter not greater than \(r_k = (C+1)\lambda^{-\min{\{p,q\}}k}\).
    According to the formula of Nussbaum~\cite{Nussbaum70} (ball measure of non-compactness) this implies that the essential spectral radius of \(\cL  : \cB_{N}^{p,q} \to \cB_{N}^{p,q}\) is not greater than \(\lim_{k\to\infty} {r_k}^{1/k} = \lambda^{-\min{\{p,q\}}}\).
\end{proof}

%%%%%%%%%%%%%%%%%%%%%%
\section{Self-similar structure of the spectrum}
\label{sec:self-similar}
%%%%%%%%%%%%%%%%%%%%%%

This section is devoted to the part of the argument which allows us to derive the full spectrum of \(\cL : \cB_{N}^{p,q} \to \cB_{N}^{p,q}\), \(N\neq 0\) from the peripheral spectrum.
More precisely, we will show that,
\begin{enumerate}
    \item The exterior part of the spectrum is contained within the spectrum of \(\cL\) restricted to the kernel of \(\dv\),
    \item The remainder of the discrete spectrum consists of multiple scaled copies of the exterior part.
\end{enumerate} 
As before, \(M = \lquotient{\Gamma}{\bH}\) and \(N\in \bZ \setminus \{0\}\) are considered chosen and fixed for the entire section.

We will repeatedly take advantage of the lemma of Baladi \& Tsujii~\cite[Lemma A.1]{BT08} which tells us when the point spectrum of an operator considered on two different Banach spaces will coincide.
Applying \cite[Lemma A.1]{BT08} in the present context (using Lemma~\ref{lem:distributions} and the fact that all \(\cB_{N}^{p,q}\) share a dense subset by construction) we have the following. 

\begin{lemma}
    \label{lem:BL}
    Let \(\cB_{N}^{p,q}\) and \(\cB_{N}^{p',q'}\) be two anisotropic Banach spaces as defined in Section~\ref{sec:anisotropic}.
    Suppose that the restrictions of \(\cL\) to \(\cB_{N}^{p,q}\) and \(\cB_{N}^{p',q'}\) are bounded operators whose essential spectral radii are both strictly smaller than \(\rho>0\).
    Then the eigenvalues of \(\cL : \cB_{N}^{p,q} \to \cB_{N}^{p,q}\) and \(\cL : \cB_{N}^{p',q'} \to \cB_{N}^{p',q'}\) in \(\{z\in \bC : \abs{z} > \rho\}\) coincide.
    Furthermore the corresponding generalized eigenspaces coincide and are contained in the intersection \(\cB_{N}^{p,q} \cap \cB_{N}^{p',q'}\).
\end{lemma}

The following tells us that the outer part of the spectrum is given by the spectrum of \(\cL\) restricted to the kernel of \(\dv\).

\begin{lemma}
    \label{lem:spec-ker}
    Let \(p,q\in \bN\), \(N\neq 0\).
    The spectrum of \(\cL : \cB_{N}^{p,q} \to \cB_{N}^{p,q}\) restricted to \(\{z\in \bC:\abs{z} > \lambda^{-1}\}\) is contained within the spectrum of \({\left.\cL\right|}_{\ker_N(\dv)}\).
\end{lemma}

\begin{proof}
    Suppose that \(z\) is an eigenvalue of \(\cL : \cB_{N}^{p,q} \to \cB_{N}^{p,q}\) and that \(\abs{z} > \lambda^{-1}\).
    Consequently there exists \(h \in \cB_{N}^{p,q}\) such that \(\cL h = z h\).
    By Lemma~\ref{lem:XL-LX},
    \[
       \cL (\dv h) = \lambda \dv \cL (h) = \lambda z \cdot \dv h.
    \]
    This means either that \( \lambda z \) is an eigenvalue of \(\cL : \cB_{N}^{p-1,q} \to \cB_{N}^{p-1,q}\)  or that \(\dv h = 0\).
    However \(\abs{\lambda z}>1\) but the spectral radius of \(\cL\) is at most 1 (Lemma~\ref{lem:spectrum}).
    This means that \(\dv h = 0\), i.e., \(h \in \ker_N(\dv)\).
\end{proof}

We now consider the remainder of the discrete spectrum.

\begin{lemma}
    \label{lem:spec-induc}
    Let \(p,q, k\in \bN\) with \(p,q > k\).
    The spectrum of \(\cL : \cB_{N}^{p,q} \to \cB_{N}^{p,q}\), \(N\neq 0\), restricted to \(\{z\in \bC : \lambda^{-(k+1)}< \abs{z}  \leq \lambda^{-k}\}\) is contained within \(\{z\in \bC: \lambda^k z \in \spec{{\left.\cL\right|}_{\ker_N(\dv)}} \}\).
\end{lemma}

\begin{proof}
    We take advantage of the fact that the point spectrum is independent of the Banach space (Lemma~\ref{lem:BL}).
    We will prove, by induction, that for all \(k\in \bN_0\),
    \begin{equation*}
        \text{if \(z \in \spec{\cL}\) and \(\lambda^{-(k+1)}< \abs{z}  \leq \lambda^{-k}\)},
        \quad
        \text{then \(\lambda^k z \in \spec{{\left.\cL\right|}_{\ker_N(\dv)}}\)}.
    \end{equation*}
    The case \(k=0\) is given by Lemma~\ref{lem:spec-ker}.
    We now suppose the statement is already proven for \(k\).
    Suppose that \(z\) is an eigenvalue of \(\cL\) and that \(\lambda^{-(k+2)}< \abs{z}  \leq \lambda^{-(k+1)}\).
    Let \(h \neq 0\) be such that \(\cL h = z h\).
    By Lemma~\ref{lem:XL-LX},
    \[
        \cL (\dv h) = \lambda \dv \cL (h) = \lambda z \cdot \dv h.   
    %\lambda \dv \cL h_z = \lambda z (\dv h_z) =  \cL (\dv h_z).
    \]
    This means that, either \( \lambda z \) is an eigenvalue of \(\cL\) or \(\dv h = 0\).
    In the second case, \(h \in \ker(\dv)\) so by Lemma~\ref{lem:spec-ker}, \(\abs{z} > \lambda^{-1}\) but this contradicts the fact that \(\abs{z} \leq \lambda^{-(k+1)} \leq \lambda^{-1}\).
    This means that the first case is the only possibility.
    Consequently we know that \(\lambda z \in \spec{\cL}\) and  \(\lambda^{-(k+1)}< \abs{\lambda z}  \leq \lambda^{-k}\).
    We then apply the inductive hypothesis and conclude.
\end{proof}

Next we take advantage of the operator \(\dw\) in order to upgrade Lemma~\ref{lem:spec-induc} to equality.
The technique we will use has some similarity to the argument used for pseudo-Anosov maps \cite{FGL19}.
It is very different to the technique of Flaminio \& Forni~\cite[A.3]{FF06} who took advantage of a basis of Hermite functions to write a formal inverse of \(\dv\) and then further argue for the ``iterated invariant distributions'' (see also \cite{Forni20}).

\begin{lemma}
    \label{lem:iterated-spec}
    Let \(p,q, k_0\in \bN\) with \(p,q > k_0\).
    The spectrum of \(\cL : \cB_{N}^{p,q} \to \cB_{N}^{p,q}\), \(N\neq 0\), restricted to \(\{z\in \bC: \lambda^{-k_0} < \abs{z}\}\) is equal to 
    \[
        \{z\in \bC : \exists k \in \{0,\ldots,k_0\}, \lambda^k z \in \spec{{\left.\cL\right|}_{\ker_N(\dv)}} \}.
    \]
\end{lemma}

\begin{proof}
    In light of Lemma~\ref{lem:spec-induc}, we need only prove that \(z\) in the spectrum implies that \(\lambda^{-1}z\) is also in the spectrum.
    For this we take advantage of the fact, established in Lemma~\ref{lem:ker-W}, that  \(\dw : \cB_{N}^{p,q} \to \cB_{N}^{p,q+1}\) is injective.
    Let \(h \neq 0\) be such that \(\cL h = z h\).
    By Lemma~\ref{lem:XL-LX},
    \[
    \cL (\dw h)=\lambda^{-1} \dw \cL (h) = \lambda^{-1} z\cdot \dw h.
    % \lambda^{-1} \dw \cL h_z = \lambda^{-1} z (\dw h_z) =  \cL (\dw h_z).
    \]
    Here again we take advantage of the fact that the point spectrum is independent of the Banach space~\cite[Lemma A.1]{BT08}.
    By Lemma~\ref{lem:ker-W},  \(\dw h \neq 0\), so \(\lambda^{-1} z\) is in the spectrum.
\end{proof}

%%%%%%%%%%%%%%%%%%%%%%
\section{Peripheral spectrum}
\label{sec:inv-dist}
%%%%%%%%%%%%%%%%%%%%%%

This section is devoted to determining the spectrum of the composition operator when restricted to the kernel of \(\dv\).
The motive for this is that the spectrum of this restricted operator coincides with the part of the spectrum of the full operator which lies outside a particular radius.
In this section we use Assumption~\ref{ass:inv-dist-in-Bpq}.
Let $\ker_N(\dv) = \{h \in  \cB_{N}^{p,q} : \dv h = 0\}$.
We will see that $\ker_N(\dv)$ is finite dimensional.
Moreover we will see that \(\ker_N(\dv)\) does not depend on \(p,q\), so justifying this choice of notation.

Recall that, according to Lemma~\ref{lem:distributions}, the anisotropic spaces are continuously embedded in the space of distributions.
Additionally, by Assumption~\ref{ass:inv-dist-in-Bpq}, 
the space \(\cB_{N}^{p,q}\) is sufficiently large so that we see the ``invariant distributions'' which were identified by Flaminio \& Forni~\cite{FF06}.

\begin{lemma}
    \label{lem:dim-IN}
    For any $N \neq 0$, 
    \(\dim\ker_N(\dv) = K\abs{N}\). 
\end{lemma}

\begin{proof}
    Let \(\cJ_N\) denote the space of \(\dv\)-invariant distributions, i.e., the set of \(D \in \cC^\infty(M)'\) such that \(D(\dv h) = 0\) for all \(h \in \cC^\infty\) and supported on \(\cC^\infty_N\).
    
    We know  that \(\iota (\ker_N(\dv)) \subset \cJ_N\) (Lemma~\ref{lem:distributions}) and that  \(\cJ_N \subset \iota (\ker_N(\dv))\) (Assumption~\ref{ass:inv-dist-in-Bpq}). It is known~\cite[Proposition 4.4]{FF06} that \(\dim\cJ_N = K\abs{N} \) and consists of distributions of Sobolev order \(\frac{1}{2}\).
    Consequently, \(\iota (\ker_N(\dv)) = \cJ_N\) and we prove the statement.
\end{proof}

\begin{lemma}
    \label{lem:bounded}
    There exists \(C > 0\) such that, for all \(n\in \bN\), \(h \in \ker_N(\dv)\),
    \[
        C^{-1} \norm{h}_{p,q} \leq \lambda^{-n/2}\norm{\cL^n h}_{p,q} \leq C \norm{h}_{p,q}.
    \]
\end{lemma}

\begin{proof}
    It is known \cite[Proposition 4.8]{FF06} that \(\lambda^{-1/2} \cL\) is an isometry on \(\ker_N(\dv)\) with respect to the Sobolev norm \(\norm{\cdot}_{s}\) used in that reference.
    I.e., for all \(n\in \bN\), \(h \in \ker_N(\dv)\), \(\norm{\cL^n h}_{s} = \lambda^{n/2}\norm{h}_{s}\).
    Since \(\ker_N(\dv)\) is finite dimensional (Lemma~\ref{lem:dim-IN}), all norms are equivalent, and so there exists \(C>0\) such that, for all \(n\in \bN\), \(h \in \ker_N(\dv)\), \(C \norm{\cL^n h}_{s} \leq  \norm{\cL^n h}_{p,q} \leq C^{-1} \norm{\cL^n h}_{s}\).
    Hence \(C \lambda^{n/2} \norm{h}_{s} \leq  \norm{\cL^n h}_{p,q} \leq C^{-1} \lambda^{n/2}  \norm{h}_{s}\).
\end{proof}

\begin{remark}
    \label{rem:norm-choice}
    The proof of Lemma~\ref{lem:bounded} relies heavily on the reference where the operator is shown to be an isomorphism with respect to a Sobolev norm.
    That norm is very convenient for this particular estimate whereas trying to obtain such an estimate with the present anisotropic norm appears difficult or impossible.
    On the other hand, much of the present work in the previous section is easier using the anisotropic norm.
\end{remark}

\begin{lemma}
    \label{lem:outer-spec}
    There exists a set of \(K\abs{N}\) unit complex numbers \(\{\mu_j\}_{j=1}^{K\abs{N}}\) such that the spectrum of \({\left.\cL\right|}_{\ker_N(\dv)}\) is 
    \(\{\lambda^{-\frac{1}{2}}\mu_{j}\}_{j=1}^{K\abs{N}}\), repeated according to algebraic multiplicity.
\end{lemma}

\begin{proof}
    The spectrum of \(\lambda^{-1/2} {\left.\cL\right|}_{\ker_N(\dv)}\) is contained in 
    \(\{z\in \bC:\abs{z} = 1\}\).
    In order to show this, suppose for sake of contradiction that there exists \(z\) in the spectrum such that \(\abs{z} \neq 1\).
    If \(h\) is a normalised eigenvector associated to \(z\), then \(\lambda^{-n/2}{\|\cL^n h\|}_{p,q} = \abs{z}^n\). However, for large \(n\), this contradicts Lemma~\ref{lem:bounded}.
    
    Suppose, again for sake of contradiction that there is a Jordan block. 
    Then, in Jordan normal form, the operator is written as a matrix which includes a \(k \times k\), \(k \geq 2\) block with diagonal \(\mu\) with \(\abs{\mu}=1\) and the upper diagonal entries equal to \(1\).
    The \(n\)\textsuperscript{th} iterate of this block will have diagonal entries equal to \(\mu^n\) and upper diagonal entries equal to \(n\mu^{n-1}\).
    In particular, the max norm of this matrix grows linearly. 
    Since norms are equivalent in a finite dimensional space, this again  contradicts Lemma~\ref{lem:bounded}. 

    That the sum of the algebraic multiplicities of the eigenspaces is equal to \(K\abs{N}\) follows from Lemma~\ref{lem:dim-IN}.
\end{proof}

\appendix
%%%%%%%%%%%%%%%%%%%%%%
\section{Partially hyperbolic automorphisms}
\label{sec:examples}
%%%%%%%%%%%%%%%%%%%%%%

In this section we discuss the explicit form of partially hyperbolic automorphisms.
For our present purposes we say that an automorphism \(\Phi : M \to M\) is \emph{partially hyperbolic with neutral centre} if the tangent bundle admits a splitting \(TM = E_s \oplus E_c \oplus E_u\) into one dimensional sub-bundles such that
\begin{itemize}
    \item \(\Phi\) is uniformly contracting on \(E_s\) and uniformly expanding on \(E_u\);
    \item \(E_c\) corresponds to \(\dz\) and \(\Phi\) is an orientation-preserving isometry on this bundle.
\end{itemize}

% In this section, to match with the pertinent references, we work in \emph{polarised} coordinates, i.e., \(\bH^{\pol}\) is equal to \(\bR^3\) with the group law
% \[
%     (x,y,z) * (x',y',z') = (x+x', y+y', z+z' + xy').
% \]
% \textcolor{red}{which corresponds
% directly to matrix multiplication.
%     It is well-known that \(\bH\) and \(\bH^{\pol}\) are isomorphic via the map
%     \(\bH\rightarrow \bH^{\pol}\) given by $(x, y, z) \mapsto (x, y, z + \frac{1}{2}xy)$.  
% Any lattice of \(\bH^{\pol}\) is also isomorphic to $\Gamma_k=\{(x,y,z)\in \bH^{\pol} \mid x,y\in \mathbb{Z}, z\in \frac{1}{k}\mathbb{Z}\}$, where $k$ is a positive integer.%We choose $k =1$ for convenience.
%}
We recall the following result which describes the structure of such automorphisms on Heisenberg nilmanifolds.

\begin{lemma}[Shi~\cite{Shi14a,Shi14b}]
    \label{lem:auto-form}
    Let \(M = \lquotient{\Gamma_1}{\bH}\).
    The function \(\Phi : M \to M\) is an automorphism, partially hyperbolic with neutral centre, if and only if, it has the form,
    \(\Phi = \exp \circ \, \phi \circ \exp^{-1}\) where
    %\footnote{We identify the Lie algebra with \(\bR^3\) using the basis \(\{X,Y,Z\}\) where, \[X = \left(\begin{smallmatrix} 0 & 1 & 0 \\ 0 & 0 & 0 \\ 0 & 0 & 0 \end{smallmatrix}\right), \quad Y = \left(\begin{smallmatrix} 0 & 0 & 0 \\ 0 & 0 & 1 \\ 0 & 0 & 0 \end{smallmatrix}\right), \quad Z = \left(\begin{smallmatrix} 0 & 0 & 1 \\ 0 & 0 & 0 \\ 0 & 0 & 0 \end{smallmatrix}\right).\]}
    \begin{equation}
        \label{eq:lie-auto}
        \phi = \begin{pmatrix}
            a                    & b                 & 0 \\
            c                    & d                 & 0 \\
            \tfrac{ac}{2} + \ell & \tfrac{bd}{2} + m & 1
        \end{pmatrix}
    \end{equation}
    for some \(a,b,c,d,\ell,m \in \bZ\) such that \(\det \left(\begin{smallmatrix} a & b \\ c & d \end{smallmatrix}\right) = 1\) and \(\left(\begin{smallmatrix} a & b \\ c & d \end{smallmatrix}\right)\) has eigenvalues \(\lambda, \lambda^{-1}\) for some \(\lambda>1\).
    Equivalently,
    \begin{equation}
        \label{eq:def-Phi}
        \begin{aligned}
            \Phi (x,y,z) & = \left( ax + by, cx + dy, z + \tau(x,y) \right),                                                                        \\
            \tau(x,y)    & = \tfrac{ac}{2} x^2 + bc xy + \tfrac{bd}{2} y^2 + \left(\tfrac{ac}{2} + \ell\right)x +  \left(\tfrac{bd}{2} + m\right)y.
        \end{aligned}
    \end{equation}
\end{lemma}

%\begin{proof}[Sketch of proof]
 %   For full details of these calculations consult the work of Shi~\cite{Shi14a,Shi14b}; here we present an overview.
 %   There is a one-to-one correspondence between the automorphisms of the Lie algebra and the automorphisms of \(\bH\) \cite[Proposition 1.21]{Folland89}.
 %   The Lie bracket must be preserved by any automorphism of the Lie algebra which gives several restrictions on the possible form.
 %   The matrix~\eqref{eq:def-Phi} represents the associated automorphism of the Lie algebra.
 %   Applying the exponential function, the formula of the automorphism on \(\bH\) can then be deduced.
 %   Further observing that the automorphism must preserve the lattice \(\Gamma\) and be an orientation preserving isometry in \(Z\) gives the final restrictions on the form of the automorphism.
%\end{proof}

%An automorphism associated to \( \left(\begin{smallmatrix} 2 & 1 \\ 1 & 1 \end{smallmatrix}\right)\), the classic Anosov map of the torus, is given by Lemma~\ref{lem:auto-form} with the coefficients \(a=2, b=c=d=1, l=m=0\).

Let \(\Phi : M \to M\) be a partially hyperbolic automorphism as above.
%partially hyperbolic with neutral centre.
I.e., \(a,b,c,d,\ell,m \in \bZ\) are fixed in such a way that they satisfy the requirements detailed in Lemma~\ref{lem:auto-form}.
Let \(\left( \begin{smallmatrix} \alpha \\ \beta \end{smallmatrix}\right)\), \(\left( \begin{smallmatrix} \beta \\ -\alpha \end{smallmatrix}\right)\) be the normalised eigenvectors of \(\left(\begin{smallmatrix} a & b \\ c & d \end{smallmatrix}\right) \).
Without loss of generality, we suppose that \(\alpha, \beta > 0\) are such that \(\alpha^2 + \beta^2 = 1\) and
\[
    \left(\begin{smallmatrix} a & b \\ c & d \end{smallmatrix}\right) \left( \begin{smallmatrix} \alpha \\ \beta \end{smallmatrix}\right) = \lambda \left( \begin{smallmatrix} \alpha \\ \beta \end{smallmatrix}\right),
    \quad
    \left(\begin{smallmatrix} a & b \\ c & d \end{smallmatrix}\right) \left( \begin{smallmatrix} \beta \\ -\alpha \end{smallmatrix}\right) = \lambda^{-1} \left( \begin{smallmatrix} \beta \\ -\alpha \end{smallmatrix}\right).
\]
Let
\begin{equation}\label{eq:frame}
    V = \alpha X + \beta Y + \gamma Z, \quad
    W = -\beta X + \alpha Y + \gamma' Z
\end{equation}
where
\[
    \begin{aligned}
        \gamma  & = \tfrac{1}{\lambda - 1} \left(\alpha(\tfrac{ac}{2} + \ell) + \beta(\tfrac{bd}{2} + m)\right),      \\
        \gamma' & = \tfrac{1}{1 - \lambda^{-1}} \left(\beta(\tfrac{ac}{2} + \ell) - \alpha(\tfrac{bd}{2} + m)\right).
    \end{aligned}
\]

\begin{lemma}
    Let \(V\) and \(W\) be elements of Lie algebra of $\bH$, defined as above.
    Then \(\{V, W, Z \}\) is a Heisenberg frame, which satisfies the commutation relation \eqref{eq:comm}.
    Moreover,
    \[
        \Phi_{*} V = \lambda^{-1} V,
        \quad
        \Phi_{*} W = \lambda W.
    \]
\end{lemma}
\begin{proof}
    The commutation relations follow from the commutation relations for \(X,Y,Z\), together with the chosen normalization.
    The two equalities are verified using the matrix form  of the automorphism of the Lie algebra~\eqref{eq:lie-auto}.
\end{proof}

%%%%%%%%%%%%%%%%%%%%%%
\section*{Acknowledgements}
%%%%%%%%%%%%%%%%%%%%%%

\begin{small}
    We are grateful to Giovanni Forni and Carlangelo Liverani for several enlightening discussions.
    We are particularly grateful to Lucia Simonelli who, several years ago, searched out details, figured out the calculations and explained much of this material to us. 
    M.K. also thanks Dalia Terhesiu for several helpful comments and conversations about the draft.
    We greatly appreciate the referees for their careful reading of the first version and precise suggestions which have drastically improved the quality of this manuscript.

    The work was partially supported by PRIN Grant ``Regular and stochastic behaviour in dynamical systems" (PRIN 2017S35EHN).
    O.B. acknowledges the MIUR Excellence Department Projects awarded to the Department of Mathematics, University of Rome Tor Vergata, CUP E83C18000100006, CUP E83C2300033000s6.
    M.K. was partially supported by UniCredit Bank R\&D group through the ``Dynamics and Information Theory Institute'' at the Scuola Normale Superiore, foundations for ``The Royal Swedish Academy of Sciences'' (MA2024-0038, MG2024-0013), and Carl Trygger's Foundation for Scientific Research, CTS 23:3036.

    This research is part of the O.B.'s activity within the UMI
    Group ``DinAmicI'' and the INdAM group GNFM. 
    M.K. acknowledges the Center of Excellence ``Dynamics, mathematical analysis and artificial intelligence'' at the Nicolaus Copernicus University in Toruń and University of Rome Tor Vergata for their hospitality.
\end{small}

%%%%%%%%%%%%%%%%%%%%%
\section*{Statements and Declarations}
%%%%%%%%%%%%%%%%%%%%%

\begin{small}
   The authors have no relevant financial or non-financial interests to disclose.
   Data sharing is not applicable to this article as no datasets were generated or analysed during the current study.
\end{small}

%%%%%%%%%%%%%%%%%%%%%%

\end{document}